\documentclass[a4paper]{article}

\usepackage{amsmath}
\usepackage{amssymb}
\usepackage{amsfonts}
\usepackage{bbm,amsmath,amsthm,tikz}

\newcommand{\RR}{\mathbbm{R}}
\newcommand{\QQ}{\mathbbm{Q}}
\newcommand{\ZZ}{\mathbbm{Z}}
\newcommand{\PP}{\mathbbm{P}}

\newtheorem{theorem}{Theorem}
\newtheorem{lemma}[theorem]{Lemma}
\newtheorem{proposition}[theorem]{Proposition}
\newtheorem{definition}[theorem]{Definition}
\newtheorem{corollary}[theorem]{Corollary}
\newtheorem{conjecture}[theorem]{Conjecture}
\providecommand{\keywords}[1]{\textbf{\textit{Index terms---}} #1}
\begin{document}
\title{On the $n$-th row of the graded Betti table of an $n$-dimensional toric variety}
\author{Alexander Lemmens}
\date{}
\maketitle
\footnotetext[1]{The final publication is available at springer via http://dx.doi.org/10.1007/s10801-017-0786-y}
\begin{abstract}
We prove an explicit formula for the first non-zero entry in the $n$-th row of the graded Betti table of an $n$-dimensional projective toric variety associated to a normal polytope with at least one interior lattice point.
This applies to Veronese embeddings of $\mathbb{P}^n$.
We also prove an explicit formula for the entire $n$-th row when the interior of the polytope is one-dimensional. All results are valid over an arbitrary field $k$.
\end{abstract}
\tableofcontents

\section{Introduction}
Let $k$ be a field. In this article we study syzygies of projectively embedded toric varieties $X/k$.
More precisely, we give explicit formulas in terms of the combinatorics of the defining polytope for certain graded Betti numbers, which appear in the minimal free resolution of the homogeneous coordinate ring of $X$ as a graded module, obtained by repeatedly taking syzygies. These Betti numbers are typically gathered in the graded Betti table:
$$
\begin{array}{r|cccccc}
  & 0 & 1 & 2 & 3 & 4 & \ldots \\
\hline
0 & 1 & 0 & 0 & 0 & 0 & \dots  \\
1 & 0 & \kappa_{1,1} & \kappa_{2,1} & \kappa_{3,1} & \kappa_{4,1} & \dots  \\
2 & 0 & \kappa_{1,2} & \kappa_{2,2} & \kappa_{3,2} & \kappa_{4,2} & \dots  \\
3 & 0 & \kappa_{1,3} & \kappa_{2,3} & \kappa_{3,3} & \kappa_{4,3} & \dots  \\
\vdots & \vdots & \vdots & \vdots & \vdots & \vdots &  \\
\end{array}$$

Here $\kappa_{p,q}$ is the number of degree $p+q$ summands of the $p$-th module in the resolution. One alternatively defines $\kappa_{p,q}$ as the dimension of the Koszul cohomology space $K_{p,q}(X,O(1))$. The graded Betti table is expected to contain a wealth of geometric information and is the subject of several important open problems and conjectures. But the vast part of it is poorly understood.\\

For a number of entries an explicit formula in terms of the defining lattice polytope was known before. Examples of this can be found in \cite{bettitoricsurfaces,heringphd}. But for this paper the most relevant result is that of
Schenck, who proved \cite{schenck} that for projective toric surfaces coming from a lattice polygon with $b$ lattice boundary points $\kappa_{p,2}=0$ for all $p\leq b-3$. Hering proved in \cite[Theorem IV.20]{heringphd} using a theorem of Gallego-Purnaprajna \cite[Theorem 1.3]{gallego} that the next entry $\kappa_{b-2,2}$ is not zero. Both results were already known in the case this polygon is equal to the triangle with vertices $(0,d)$, $(d,0)$, $(0,0)$. This polygon gives the $d$-fold Veronese embedding of the projective plane, for which Loose \cite{loose} proved that the number of zeroes in the quadratic strand equals $3d-3$ (not counting $\kappa_{0,2}=0$ as a zero). This result was independently rediscovered by Ottaviani and Paoletti \cite{ottaviani}, and they generalized this to the following conjecture:
\begin{conjecture}
For the $d$-fold Veronese embedding of $n$-dimensional projective space $\kappa_{p,q}=0$ whenever $p\leq 3d-3$ and $q\geq 2$.
\end{conjecture}
This is known as property $N_p$ with $p=3d-3$. For $d>n$ this is generalized by the following conjecture \cite[p.\ 643, Conjecture 7.6]{einlazarsfeld} which the authors already proved for $q=n$:
\begin{conjecture}\label{ein}
If we take a minimal free resolution of the line bundle $\mathcal{O}_{\PP^n}(b)$ on the Veronese embedding of $\PP^n$ of degree $d$ with $d\geq b+n+1$ then $\kappa_{p,q}=0$ for all $1\leq q\leq n$ and
$$p<\binom{d+q}{q}-\binom{d-b-1}{q}-q.$$
\end{conjecture}
Syzygies of Veronese embeddings are still an active area of research \cite{brunsconca,EinErmanLazarsfeld,greco,park,rubei}. For a short introduction to syzygies and to toric varieties we refer the reader to the next section.\\

We will not prove this conjecture, but we will prove an explicit formula for $\kappa_{\binom{d+n}{n}-\binom{d-b-1}{n}-n,n}$ which is the first non-zero entry on the $n$-th row.
We also prove a formula for the first non-zero entry in the $n$-th row of the Betti table of any projectively normal toric variety of dimension $n$, if this row is not zero. Note that the $n$-th row is the last non-zero row if it is not zero.
We will work over an arbitrary field $k$.
For a convex lattice polytope $\Delta$ we denote by $\Delta^{(1)}$ the convex hull of the lattice points in the topological interior of $\Delta$.
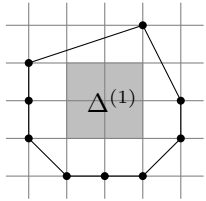
\begin{center}
\begin{tikzpicture}
\draw[lightgray, fill=lightgray] (.5,.5) -- (1.5,.5) -- (1.5,1.5) -- (.5,1.5) -- (.5,.5);
\draw[gray] (-.3,0) -- (2.3,0);
\draw[gray] (-.3,1) -- (2.3,1);
\draw[gray] (-.3,2) -- (2.3,2);
\draw[gray] (-.3,.5) -- (2.3,.5);
\draw[gray] (-.3,1.5) -- (2.3,1.5);
\draw[gray] (0,-.3) -- (0,2.3);
\draw[gray] (1,-.3) -- (1,2.3);
\draw[gray] (2,-.3) -- (2,2.3);
\draw[gray] (.5,-.3) -- (.5,2.3);
\draw[gray] (1.5,-.3) -- (1.5,2.3);
\draw (.5,0) -- (1.5,0) -- (2,.5) -- (2,1) -- (1.5,2) -- (0,1.5) -- (0,.5) -- (.5,0);
\fill[black] (.5,0) circle (1.5 pt);
\fill[black] (1,0) circle (1.5 pt);
\fill[black] (1.5,0) circle (1.5 pt);
\fill[black] (2,.5) circle (1.5 pt);
\fill[black] (2,1) circle (1.5 pt);
\fill[black] (1.5,2) circle (1.5 pt);
\fill[black] (0,1.5) circle (1.5 pt);
\fill[black] (0,1) circle (1.5 pt);
\fill[black] (0,.5) circle (1.5 pt);
\node at (1.1,1) {$\Delta^{(1)}$};
\end{tikzpicture}
\end{center}
We are now ready to formulate our main result:
\begin{theorem}\label{maintheorem}
Let $X$ be a toric variety coming from an $n$-dimensional normal polytope $\Delta\subseteq\RR^n$. Let $\Delta^{(1)}$ be the interior polytope of $\Delta$. Let $S=\Delta\cap\ZZ^n$, $T=\Delta^{(1)}\cap\ZZ^n$ and $N=\#S$, $N^{(1)}=\#T$.
\begin{itemize}
\item If $p<N-N^{(1)}-n$ then $\kappa_{p,n}=0$.
\item If $\Delta^{(1)}$ is not 1-dimensional then $\kappa_{N-N^{(1)}-n,n}=\binom{t+N^{(1)}-2}{N^{(1)}-1}$ where $t$ is the number of translations of $T$ that are contained in $S$. (If $N^{(1)}=0$ then $\kappa_{N-N^{(1)}-n,n}=0$.)
\item If $\Delta^{(1)}$ is 1-dimensional then $\forall p\geq 0$: $\kappa_{p+N-N^{(1)}-n,n}=(p+1)\binom{N-\ell}{N^{(1)}-p-1}$ where $\ell$ is the number of lines parallel to $\Delta^{(1)}$ that are not disjoint with $S$.
\end{itemize}
\end{theorem}
The first statement actually follows from Green's linear syzygy theorem \cite[Theorem 7.1]{eisenbud} combined with Koszul duality. The second statement already appeared for $n=2$ as a conjectural formula in \cite{bettitoricsurfaces}, where more information on Koszul cohomology of toric surfaces can be found. Recall that $\kappa_{p,q}=0$ whenever $q>n$, and note that if $\Delta^{(1)}=\emptyset$ then $\kappa_{p,q}=0$ whenever $q\geq n$ as follows from \cite[Proposition IV.5 p.\ 17-18]{heringphd}.\newline
%\begin{corollary}\label{rowzero}
%The $n$-th row of the Betti table is zero if there are no interior lattice points of $\Delta$
%\end{corollary}
\begin{theorem}\label{veronese}
In the context of conjecture \ref{ein} the first non-zero entry on the $n$-th row equals
$$\kappa_{\binom{d+n}{d}-\binom{d-b-1}{n}-n,n}=\binom{\binom{b+2n+1}{n}+\binom{d-b-1}{n}-2}{\binom{d-b-1}{n}-1}.$$
\end{theorem}

%Note that this is zero if $d\leq n$ in which case the entire $n$-th row is zero.
%This confirms conjecture \ref{ein} for $q=n$ as $3d-3<\binom{n+d}{n}-\binom{d-1}{n}-n$ for all $n\geq 2$ and $d>n$.\newline
These two theorems will be proved at the end of section 2 using results from section 3.
\begin{corollary}\label{width2}
For toric surfaces coming from polygons of lattice width two, we know the entire Betti table explicitly. $$\kappa_{p,2}=\max(p-N+N^{(1)}+3,0)\binom{N-3}{p},\hspace{5 pt}\kappa_{p,1}=\kappa_{p-1,2}+p\binom{N-1}{p+1}-2A\binom{N-3}{p-1}$$
where $A=N/2+N^{(1)}/2-1$ is the area of $\Delta$. Of course $\kappa_{0,0}=1$ and everything else is zero.
\end{corollary}
The second formula comes from \cite[lemma 1.3]{bettitoricsurfaces}, the first follows directly from our theorem \ref{maintheorem}.\\
%In fact one can also prove this corollary using techniques from \cite{schreyer}. You embed the toric surface in a three-dimensional rational normal scroll, you take a minimal free graded resolution of its homogeneous coordinate ring and the graded ideal cutting out the surface, and then you take the mapping cone of the chain map between these resolutions. This gives a free minimal graded resolution of the toric surface.

Using \cite[Theorem 1.3]{betticurves} one can deduce the following formula for the graded Betti table of the canonical model of a tetragonal curve in a toric surface:
$$\kappa_{g-p-2,1}=\kappa_{p,2}=(g-p-2)\binom{g-3}{p-2}+\sum_{i=1}^2\max(p-b_i-1,0)\binom{g-3}{p},$$
where $b_1,b_2$ are the tetragonal invariants introduced by Schreyer in \cite[p.\ 127]{schreyer}, and $g$ is the genus. Actually this formula is true for all tetragonal curves as follows from the explicit minimal free graded resolution in Schreyer's article. We include this explicit formula because it is not easy to find in the literature.\\

In section 2 we explain toric varieties, syzygies, Koszul cohomology and we prove these theorems using results from section 3. We use Koszul duality \cite[p.\ 21]{nagelaprodu} which expresses Betti numbers on the $n$-th row in terms of Betti numbers on the first row of the Betti table of the Serre dual line bundle.\\

The core of the article is section 3 where we construct an explicit basis for the last non-zero entry on the first row of the graded Betti table of any graded module of the form $\bigoplus_{q\geq 0}H^0(qL+M)$ for line bunldes $L,M$ with $H^0(M)=0$, $H^0(L)\neq 0$, $H^0(L+M)\neq 0$ on any normal projective toric variety. This comes down to constructing a basis of the kernel of the map
$$\bigwedge^pH^0(L)\otimes H^0(L+M)\rightarrow\bigwedge^{p-1}H^0(L)\otimes H^0(2L+M)$$
with $p=\dim H^0(L+M)-1$. Theorems \ref{maintheorem} and \ref{veronese} can then be proved from results from section 3, namely theorem \ref{Tleqp} (which actually also follows from Green's linear syzygy theorem \cite[Theorem 7.1]{eisenbud}), corollary \ref{corollary} and theorem \ref{1dim}.
\subsubsection*{Acknowledgements}
This article is part of my Ph.D thesis which is funded by the Research Foundation Flanders (FWO). It was my colleagues Wouter Castryck and Filip Cools who noticed patterns in Betti tables of certain toric surfaces, which motivated me to find an explicit basis. I am also grateful to the referee for carefully reading my article and making many useful suggestions. I also want to thank Milena Hering for bringing the article \cite{einlazarsfeld} to my attention.
\section{Toric varieties and graded Betti tables}
\subsection*{Projectively normal toric varieties}
We work over an arbitrary field $k$. By lattice points we mean points of $\ZZ^n$. Projective toric varieties are built out of polytopes $\Delta\subseteq\RR^n$ that are the convex hull of a finite set of lattice points. This works as follows. Suppose $\Delta$ is $n$-dimensional and let $P_1,\ldots,P_{N_\Delta}$ be a list of all lattice points of $\Delta$, we define an embedding
$$
\phi_\Delta:(\mathbb{A}^1\backslash\{0\})^n\rightarrow\mathbb{P}^{N_\Delta-1}:
(\lambda_1,\ldots,\lambda_n)\mapsto\Big(\prod_{i=1}^n\lambda_i^{P_{1,i}}:\ldots:\prod_{i=1}^n\lambda_i^{P_{N_\Delta,i}}\Big),
$$
where $P_{j,i}$ is the $i$-th coordinate of $P_j$. Let $X_\Delta$ be the closure of the image of $\phi_\Delta$.
If it happens that $a\Delta\cap\ZZ^n+b\Delta\cap\ZZ^n=(a+b)\Delta\cap\ZZ^n$ for all positive integers $a,b$, then the polytope is called normal. In this case the projective toric variety corresponding to $\Delta$ is just $X_\Delta$ and it is projectively normal.\vspace{4 pt}\newline
\textbf{Example.}\newline
The $d$-fold Veronese embedding of projective space is given by a polytope of the following form:
$$\Delta=\{(x_1,\ldots,x_n)\in(\RR_{\geq 0})^n|x_1+\ldots+x_n\leq d\}$$
This gives the Veronese embedding of $\mathbb{P}^n$ into $\mathbb{P}^{N-1}$ where $N=\#\Delta\cap\ZZ^n=\binom{n+d}{d}$.
For instance if $n=2$ and $d=2$ we get the embedding
$$\mathbb{A}^1\backslash\{0\}\times\mathbb{A}^1\backslash\{0\}\rightarrow\mathbb{P}^5:\hspace{5 pt}(x,y)\mapsto(x^2:xy:x:y^2:y:1).$$
The monomials $x^2$, $xy$, $x$, $y^2$, $y$, $1$ correspond to the lattice points of the triangle $\Delta$ with vertices $(2,0),(0,2),(0,0)$.
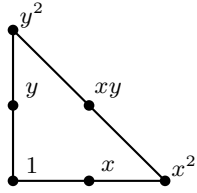
\begin{center}
\begin{tikzpicture}
\draw [thick] (0,0) -- (2,0) -- (0,2) -- (0,0);
\fill[black] (0,0) circle (2 pt);
\fill[black] (1,0) circle (2 pt);
\fill[black] (2,0) circle (2 pt);
\fill[black] (0,1) circle (2 pt);
\fill[black] (0,2) circle (2 pt);
\fill[black] (1,1) circle (2 pt);
\node at (0.25,0.2) {\small $1$};
\node at (2.25,0.2) {\small $x^2$};
\node at (0.25,2.2) {\small $y^2$};
\node at (1.25,0.2) {\small $x$};
\node at (1.25,1.2) {\small $xy$};
\node at (0.25,1.2) {\small $y$};
\end{tikzpicture}
\end{center}
When taking the Zariski closure of the image this corresponds to the standard Veronese embedding
$$\mathbb{P}^2\rightarrow\mathbb{P}^5:\hspace{5 pt}(x:y:z)\mapsto(x^2:xy:xz:y^2:yz:z^2).$$
If $\Delta$ is not normal then one can still take integer multiples $q\Delta$, $q\geq 1$, which will be normal for sufficiently large $q$. One can then associate to $\Delta$ the toric variety $X_{q\Delta}$ where $q$ is large enough so that $q\Delta$ is normal. This variety does not depend on $q$ (but its embedding does). However for simplicity we will restrict to the case when $\Delta$ is normal.
The homogeneous coordinate ring of $X_\Delta$ is given by
$$\bigoplus_{q\geq 0}V_{q\Delta}=\bigoplus_{q\geq 0}H^0(X,qL),$$
where by $V_{q\Delta}$ we mean the vector space spanned by the monomials (possibly with negative exponents) $x_1^{i_1}\ldots x_n^{i_n}$ corresponding to lattice points $(i_1,\ldots,i_n)\in q\Delta$. By $L$ we mean the very ample line bundle coming from the embedding into projective space.
\subsection*{Graded Betti tables}
Given any projective variety $X$ with homogeneous coordinate ring $R=\bigoplus_{q\geq 0}$ $H^0(X,qL)$ we can consider the graded Tor modules $$\text{Tor}_{S^*H^0(X,L)}^i(R,k).$$ Note that $R$ is a graded module over the symmetric algebra $S^*H^0(X,L)$. These graded Tor modules can be computed either by taking a graded free resolution of $R$ (syzygies) or by taking a graded free resolution of $k$ (Koszul cohomology). We will mainly work with the latter. The graded Betti table is a table of non-negative integers $\kappa_{p,q}$, in the $p$-th column and the $q$-th row, where $p,q\geq 0$. They are defined as the dimension over $k$ of the degree $p+q$ part of $\text{Tor}_{S^*H^0(X,L)}^p(R,k)$. In general the table has the following shape:
\begin{equation}
\begin{array}{r|cccccc}
  & 0 & 1 & 2 & 3 & 4 & \ldots \\
\hline
0 & 1 & 0 & 0 & 0 & 0 & \dots  \\
1 & 0 & \kappa_{1,1} & \kappa_{2,1} & \kappa_{3,1} & \kappa_{4,1} & \dots  \\
2 & 0 & \kappa_{1,2} & \kappa_{2,2} & \kappa_{3,2} & \kappa_{4,2} & \dots  \\
3 & 0 & \kappa_{1,3} & \kappa_{2,3} & \kappa_{3,3} & \kappa_{4,3} & \dots  \\
\vdots & \vdots & \vdots & \vdots & \vdots & \vdots &  \\
\end{array}\nonumber
\end{equation}
\textbf{Example.}\newline
When $\Delta$ is the convex hull of $(2,0),(0,2),(0,0)$ we have the minimal graded free resolution of $R$:
$$0\longrightarrow F_3\overset{d_3}{\longrightarrow}F_2\overset{d_2}{\longrightarrow} F_1\overset{d_1}{\longrightarrow} S^*V_{\Delta\cap\ZZ^2}\overset{d_0}{\longrightarrow} R.$$
Here $V_{\Delta\cap\ZZ^2}$ is the vector space spanned by the monomials $x^2,xy,x,y^2,y,1$ corresponding to the lattice points of $\Delta$. The symmetric algebra $S^*V_{\Delta\cap\ZZ^2}$ is the polynomial ring in 6 variables $x_{(2,0)},x_{(1,1)},x_{(1,0)},x_{(0,2)},x_{(0,1)},x_{(0,0)}$ and is the homogeneous coordinate ring of $\mathbb{P}^5$. The image of $d_0$ corresponds to the ideal cutting out the Veronese surface. This ideal is generated by six elements:
\begin{align}
&x_{(2,0)}x_{(0,2)}-x_{(1,1)}^2,\hspace{25 pt}x_{(2,0)}x_{(0,0)}-x_{(1,0)}^2,\hspace{25 pt}x_{(0,2)}x_{(0,0)}-x_{(0,1)}^2,\nonumber \\ &x_{(2,0)}x_{(0,1)}-x_{(1,0)}x_{(1,1)}\text{, }x_{(1,0)}x_{(0,2)}-x_{(0,1)}x_{(1,1)}\text{, }x_{(1,1)}x_{(0,0)}-x_{(1,0)}x_{(0,1)}.\nonumber
\end{align}
These constitute a minimal set of generators of the ideal. So $F_1$ is a free graded module of rank six over the polynomial ring $S^*V_{\Delta\cap\ZZ^2}$ where the basis elements all have degree two and are mapped by $d_0$ to the generators of the ideal. This makes sure that $d_0$ is a degree-preserving morphism of modules. This means that $\kappa_{1,1}=6$.\newline
The image of $d_1$ consists of the relations between these generators, called syzygies. It turns out that there is a minimal generating set of eight syzygies of degree 3, so that $F_2$ is a rank 8 graded free module where the basis elements have degree 3. So $\kappa_{2,1}=8$. It turns out that $F_3$ has rank 3 where the basis elements have degree 4. This gives the graded Betti table:
\begin{equation}
\begin{array}{r|cccccc}
  & 0 & 1 & 2 & 3 & 4 & \ldots \\
\hline
0 & 1 & 0 & 0 & 0 & 0 & \dots  \\
1 & 0 & 6 & 8 & 3 & 0 & \dots  \\
2 & 0 & 0 & 0 & 0 & 0 & \dots  \\
\vdots & \vdots & \vdots & \vdots & \vdots & \vdots &  \\
\end{array}\nonumber
\end{equation}
Note that $\kappa_{0,0}=1$ because the polynomial ring $S^*V_{\Delta\cap\ZZ^2}$ is a free module of rank one over itself with the monomial 1 as a generator.
%\begin{lemma}\label{auslander}
%For a toric variety associated to an $n$-dimensional normal polytope $\Delta$ with $N$ lattice points $\kappa_{p,q}=0$ whenever $p\geq N-n$.
%\end{lemma}
%This follows from two facts: one is the graded version of the Auslander-Buchsbaum formula \cite[A.2.15]{eisenbud} which says that the length of the minimal free graded resolution of a graded module $M$ over a polynomial ring in $N$ variables is $N$ minus the depth of $M$. We apply this to $R=\bigoplus_{q\geq 0}V_{q\Delta\cap\ZZ^n}$ as a module over $S^*V_{\Delta\cap\ZZ^n}$ which is (isomorphic to) a polynomial ring in $b$ variables. The other fact we need is that normal toric varieties are Cohen-Macaulay (\cite[p.\ 415]{coxlittleschenck}), meaning that the depth is equal to the dimension. Applying this to the affine toric variety $\spec(R)$ we see that the depth is $n+1$, so the length of the minimal free graded resolution is $N-n-1$.
\subsection*{Koszul cohomology}
Let $L,N$ be line bundles on a complete variety $X$. Let $S^*V=S^*H^0(X,L)$ be the symmetric algebra over $H^0(X,L)$, then $\bigoplus_{q\geq 0}H^0(X,qL+N)$ is a graded module over $S^*V$. We define the Koszul cohomology space $K_{p,q}(X,N,L)$ as the homology of the following sequence:
\begin{align}
\bigwedge^{p+1}V\otimes H^0(X,(q-1)L+N)&\overset{\delta_{p+1}}{\longrightarrow}\bigwedge^pV\otimes H^0(X,qL+N)\nonumber \\
&\overset{\delta_p}{\longrightarrow}\bigwedge^{p-1}V\otimes H^0(X,(q+1)L+N)\nonumber
\end{align}
where $\delta_p(v_1\wedge\ldots\wedge v_p\otimes w)=\sum_{i=1}^p(-1)^iv_1\wedge\ldots\wedge\hat{v_i}\wedge\ldots\wedge v_p\otimes(v_iw)$. The $\hat{v_i}$ indicates that $v_i$ is removed from the wedge product. When $N=0$ we write $K_{p,q}(X,L)=K_{p,q}(X,0,L)$. We denote the dimension of $K_{p,q}(X,N,L)$ (resp. $X_{p,q}(X,L)$) by $\kappa_{p,q}(X,N,L)$ (resp. $\kappa_{p,q}(X,L)$). If $L$ is the very ample line bundle coming from a projective embedding this agrees with our earlier definition of $\kappa_{p,q}$ using syzygies.\vspace{4 pt}\newline
\textbf{Example.}\newline
For our Veronese example with $n=d=2$ we will construct an explicit element of the cohomology space $K_{2,1}(X,L)$:
$$y^2\wedge xy\otimes xz-y^2\wedge yz\otimes x^2+xy\wedge yz\otimes yx\in\bigwedge^2 V_\Delta\otimes V_\Delta,$$
which is in the kernel of $\delta_2$, so it defines an element of $K_{2,1}(X,L)$.\\

We now turn to the proof of theorem 1 and corollary 2. To any $n$-dimensional convex lattice polytope $\Delta\subseteq\RR^n$ one can associate the inner normal fan $\Sigma$ \cite[p.\ 321]{coxlittleschenck} whose rays $\rho$ are in one-to-one correspondence with the facets of $\Delta$ and the torus-invariant prime divisors $D_\rho$.
In general for any torus-invariant divisor $D=\sum_{\rho}a_\rho D_\rho$ the vector space $H^0(X,D)$ has a basis that naturally corresponds to $\{P\in\ZZ^n|\forall\rho\in\Sigma(1):\langle P,\rho\rangle\geq-a_\rho\}$ where $\Sigma(1)$ is the set of rays of the fan $\Sigma$.
Multiplication of these global sections corresponds to coordinatewise addition of lattice points. The divisor whose global sections correspond to $\Delta$ gives the very ample line bundle of our embedding into projective space. Note that in this setting nothing changes when extending the field $k$. In the next proof we assume $k$ algebraically closed. We also use in the following proof that taking the pull-back of a line bundle through a birational morphism of projective normal varieties preserves global sections. We will also use the fact that adding $\sum_{\rho\in\Sigma(1)}-D_\rho$ to a divisor $\sum_{\rho}a_\rho D_\rho$ amounts to taking the interior of the corresponding polytope $(*)$.
\begin{proof}[Proof of theorem \ref{maintheorem} using results from the next section]
This will rely on Koszul duality which requires smoothness, so let $X'\rightarrow X$ be a toric resolution of singularities as in \cite[p.\ 515-519]{coxlittleschenck} and let $K=\sum_{\rho\in\Sigma'(1)}-D_\rho$ be the canonical divisor of $X'$. Let $L$ be the line bundle on $X$ coming from its projective embedding. The pull-back of $L$ to $X'$ (which we also denote by $L$) is globally generated on $X'$ and hence nef. By Demazure vanishing \cite[p.\ 410]{coxlittleschenck} $H^i(X',qL)=0$, $\forall i,q\geq 1$. By Koszul duality \cite[p.\ 21]{nagelaprodu} we have
$$\kappa_{N-1-n-p,n}(X,L)=\kappa_{N-1-n-p,n}(X',L)=\kappa_{p,1}(X',K,L),\hspace{5 pt}\forall p\geq 0.$$
The first equality follows because $H^0(X,qL)=H^0(X',qL)$, $\forall q\geq 0$ as taking the pull-back of $L$ through $X'\rightarrow X$ preserves global sections. We claim that $\kappa_{p,1}(X',K,L)$ is the dimension of the kernel of the following map:
$$\delta:\bigwedge^{p}V_S\otimes V_T\rightarrow\bigwedge^{p-1}V_S\otimes V_{(2\Delta)^{(1)}\cap\ZZ^n}$$
where by $V_S$ (resp. $V_T$) we mean a vector space with $S$ (resp. $T$) as a basis. This is because $H^0(X',K)=0$ and $H^0(X',L+K)$ corresponds to $T=\Delta^{(1)}\cap\ZZ^n$ which we know by $(*)$. Note that the image of $\delta$ is contained in $\bigwedge^{p-1}V_S\otimes V_{S+T}$. Now all the results of the theorem follow from theorem \ref{Tleqp}, corollary \ref{corollary} and theorem \ref{1dim}, except when $p=0$ and $N^{(1)}=1$, but then $\delta=0$ and the result is easy. For the case where $\Delta^{(1)}$ is one-dimensional, note that $\ell$ (as in theorem \ref{maintheorem}) equals $N-\#X$, with $X$ as in theorem \ref{1dim}.
\end{proof}
By the same duality as in the proof of theorem \ref{maintheorem} it follows that $\kappa_{p,q}=0$ whenever $q>n$.
\begin{proof}[Proof of theorem \ref{veronese}.]
Denote by $\kappa_{p,q}(b;d)$ the dimension of the Kuszul cohomology of the graded module
$$\bigoplus_{i\geq 0}H^0(\mathcal{O}_{\PP^n}(b+id)),$$
Let $N=\binom{n+d}{n}$ be the number of lattice points in $d\Sigma$ where
$$\Sigma=\{(x_1,\ldots,x_n)\in\RR_{\geq 0}|x_1+\ldots+x_n\leq d\}$$
is the standard simplex of dimension $n$.
As in the proof of theorem \ref{maintheorem} we have duality:
$$\kappa_{p,n}(b;d)=\kappa_{N-n-p-1,1}(-b-n-1;d).$$
Strictly speaking, we cannot apply theorem \cite[p.\ 21]{nagelaprodu}, but the proof obviously generalizes and the vanishing condition will be satisfied because $H^1,\ldots,H^{n-1}$ of any line bundle on $\PP^n$ will vanish. This $\kappa_{N-n-p-1,1}(-b-n-1;d)$ is the dimension of the kernel of
$$\bigwedge^{N-n-p-1}S\otimes T\rightarrow\bigwedge^{N-n-p-2}S\otimes S+T,$$
where $T=(d-b-n-1)\Sigma\cap\ZZ^n$ and $S=d\Sigma\cap\ZZ^n$ so that $X=(b+n+1)\Sigma\cap\ZZ^n$.
Applying theorem \ref{Tleqp} we find that $\kappa_{p,n}(b;d)=0$ whenever $N-n-p-1\geq\# T=\binom{d-b-1}{n}$, or equivalently $p<\binom{n+d}{n}-\binom{d-b-1}{n}-n$. The formula for $\kappa_{p,n}(b;d)$ when $p=\binom{n+d}{n}-\binom{d-b-1}{n}-n$ follows from corollary \ref{corollary}. 
\end{proof}
\section{Combinatorial proof}
In this section we do not require our polytopes to be normal. From now on instead of working with spaces of monomials $V_S$, $V_T$ etcetera, we replace monomials $x_1^{i_1}\ldots x_n^{i_n}$ by the corresponding lattice point $(i_1,\ldots,i_n)$.
Let $S$ and $T\subseteq\ZZ^n$ with $T$ finite, let $p\geq 0$. We abusively write $S$ (resp. $T$) for a vector space with $S$ (resp. $T$) as a basis. We are interested in the kernel of the following map:
$$\delta:\bigwedge^pS\otimes T\rightarrow\bigwedge^{p-1}S\otimes(S+T):$$ $$P_1\wedge\ldots\wedge P_p\otimes Q\longmapsto\sum_{i=1}^p(-1)^iP_1\wedge\ldots\wedge\hat{P_i}\wedge\ldots\wedge P_p\otimes (Q+P_i),$$
where $\bigwedge^{-1}$ of a vector space is zero. By $Q+P_i$ we mean coordinate-wise addition in $\ZZ^n$.
\begin{definition}\label{defsupp}
Let $S$ and $T$ be finite subsets of $\ZZ^n$ and $p\geq 0$ an integer. If $x\in\bigwedge^pS\otimes T$ then $x$ can be uniquely written (up to order) as a linear combination (with non-zero coefficients) of expressions of the form $P_1\wedge\ldots\wedge P_p\otimes Q$ with $P_1,\ldots,P_p\in S$, $Q\in T$ and $P_1>\ldots>P_p$ for some total order on $S$. We define the \emph{support} of $x$, denoted $supp(x)$, as the convex hull of the set of points $P_i$ occurring in the wedge part of the terms of $x$.
\end{definition}
\begin{definition}
A \emph{lattice pre-order} on $\ZZ^n$ is a reflexive transitive relation $\leq$ on $\ZZ^n$ such that $\forall P_1,P_2\in\ZZ^n$: $P_1\leq P_2$ or $P_2\leq P_1$ and $\forall P_1,P_2,P_3\in\ZZ^n$: if $P_1\leq P_2$ then $P_1+P_3\leq P_2+P_3$. We call $\leq$ a \emph{lattice order} if it is anti-symmetric.
\end{definition}
One way to obtain a lattice pre-order is to take a linear map $L:\RR^n\rightarrow\RR$ and set $P_1\leq P_2$ if $L(P_1)\leq L(P_2)$. If the coefficients defining $L$ are linearly independent over $\QQ$ then it defines a lattice order. We write $P_1<P_2$ if $P_1\leq P_2$ and not $P_2\leq P_1$, and we write $P_1\sim P_2$ if $P_1\leq P_2$ and $P_2\leq P_1$.
In the proof of the following lemma we will use the property that for any points $P,P_M,Q,Q'\in\ZZ^n$ and any pre-order on $\ZZ^n$, if $P\leq P_M$ and $Q\leq Q'$ then either $P+Q<P_M+Q'$ or $P_M\sim P$ and $Q'\sim Q$.
\begin{lemma}\label{Pwedge}
Let $S$ and $T$ be finite subsets of $\ZZ^n$, $p\geq 1$. Let $x\in\ker\delta$ be non-zero with $\delta$ as above. Let $\leq$ be a lattice pre-order such that $supp(x)$ has a unique maximum $P_M$. Let $S'=supp(x)\backslash\{P_M\}$ and $T'$ the set of non-maximal points $Q\in T$. Finally, let $\delta':\bigwedge^{p-1}S'\otimes T'\rightarrow\bigwedge^{p-2}S'\otimes(S'+T')$ be defined analogously to $\delta$. Then there exists a non-zero $y\in\ker\delta'$ such that
$$x=P_M\wedge y+\text{ terms not containing }P_M\text{ in the }\wedge\text{ part.}$$
\end{lemma}
\begin{proof}
Write
$$x=\sum_i\lambda_iP_M\wedge P_{2,i}\wedge\ldots\wedge P_{p,i}\otimes Q_i+\text{ terms not containing }P_M\text{ in the }\wedge\text{ part,}$$
without any redundant terms. Then define
$$y=\sum_i\lambda_iP_{2,i}\wedge\ldots\wedge P_{p,i}\otimes Q_i.$$
All we have to prove now is that $y\in\ker(\delta')$. It is clear that $supp(y)\cap\ZZ^n\subseteq S'$. Let us prove that every $Q_i$ in this expression is in $T'$. We know that $Q_i\in T$. If it is not in $T'$ then it is maximal. But then, applying $\delta$, the term $-P_{2,i}\wedge\ldots\wedge P_{p,i}\otimes(P_M+Q_i)$ of $\delta(x)$ has nothing to cancel against, contradicting the fact that $\delta(x)=0$. The reason that it has nothing to cancel against is that all terms in $\delta(x)$ end with $\otimes(P+Q)$ with $P\leq P_M$ and $Q\leq Q_i$ so that $P+Q<P_M+Q_i$ unless $P\sim P_M$, and $Q\sim Q_i$. As $P_M$ is the unique maximum of $supp(x)$ $P\sim P_M$ implies $P=P_M$. But if $Q\neq Q_i$ and $P=P_M$ then we still have $P+Q\neq P_M+Q_i$, so that we have only one term of $\delta(x)$ ending on $\otimes(P_M+Q_i)$, which has nothing to cancel against.\newline
So $y$ is in the domain of $\delta'$. We now prove that $\delta'(y)=0$:
$$0=\delta(x)=-P_M\wedge\delta'(y)+\text{ terms not containing }P_M\text{ in the }\wedge\text{ part}.$$
Because terms of $P_M\wedge\delta'(y)$ cannot cancel against terms without $P_M$ in the $\wedge$ part, $\delta'(y)$ must be zero.
\end{proof}
\noindent\textbf{Example.}\newline
Our Veronese example of $n=d=2$ becomes $S=T=\{(2,0),(1,1),(1,0),$ $(0,2),(0,1),(0,0)\}$ and in the new notation the explicit cochain in $\bigwedge^2 S\otimes S$ becomes
$$x=(0,2)\wedge(1,1)\otimes(1,0)-(0,2)\wedge(0,1)\otimes(2,0)+(1,1)\wedge(0,1)\otimes(1,1).$$
If we take the pre-order coming from the linear map $L(x_1,x_2)=x_2$ then $P_M=(0,2)$ is the unique maximum of $supp(x)=\{(0,2),(1,1),(0,1)\}$ and the lemma gives $y=(1,1)\otimes(1,0)-(0,1)\otimes(2,0)$. Indeed, $x=(0,2)\wedge y+(1,1)\wedge(0,1)\otimes(1,1)$ and the last term does not have $(0,2)$ in the wedge part.
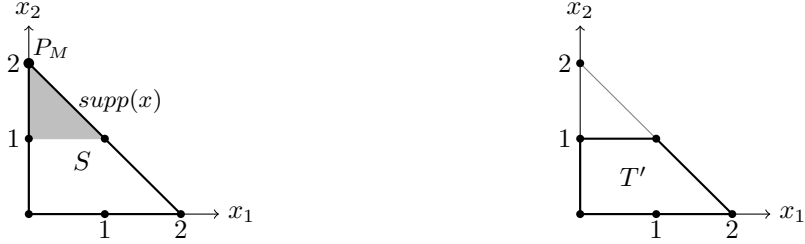
\begin{center}
\begin{tikzpicture}
\draw [<->] (2.5,0) -- (0,0) -- (0,2.5);
\draw[lightgray, fill=lightgray] (0,2) -- (0,1) -- (1,1) -- (0,2);
\draw [thick] (0,0) -- (2,0) -- (0,2) -- (0,0);
\fill[black] (0,2) circle (2 pt);
\fill[black] (0,0) circle (1.5 pt);
\fill[black] (1,0) circle (1.5 pt);
\fill[black] (2,0) circle (1.5 pt);
\fill[black] (0,1) circle (1.5 pt);
\fill[black] (1,1) circle (1.5 pt);

\node at (0.3,2.2) {\small $P_M$};
\node at (1.2,1.5) {\small $supp(x)$};
\node at (2.8,0) {$x_1$};
\node at (0,2.7) {$x_2$};
\node at (1,-.2) {1};
\node at (2,-.2) {2};
\node at (-.2,1) {1};
\node at (-.2,2) {2};
\node at (.7,.7) {$S$};
\end{tikzpicture}
\hspace{100 pt}
\begin{tikzpicture}
\draw [<->] (2.5,0) -- (0,0) -- (0,2.5);
\draw[gray] (1,1) -- (0,2);
\draw [thick] (0,0) -- (2,0) -- (1,1) -- (0,1) -- (0,0);
\fill[black] (0,2) circle (1.5 pt);
\fill[black] (0,0) circle (1.5 pt);
\fill[black] (1,0) circle (1.5 pt);
\fill[black] (2,0) circle (1.5 pt);
\fill[black] (0,1) circle (1.5 pt);
\fill[black] (1,1) circle (1.5 pt);

\node at (2.8,0) {$x_1$};
\node at (0,2.7) {$x_2$};
\node at (1,-.2) {1};
\node at (2,-.2) {2};
\node at (-.2,1) {1};
\node at (-.2,2) {2};

\node at (.7,.5) {$T'$};
\end{tikzpicture}
\end{center}
\begin{theorem}\label{Tleqp}
Notation as above. Suppose $\#T\leq p$, then $\delta:\bigwedge^pS\otimes T\rightarrow\bigwedge^{p-1}S\otimes(S+T)$ is injective.
\end{theorem}
\begin{proof}
By induction on $p$. The case $p=0$ is trivial as the domain of $\delta$ is zero (because $T=\emptyset$).\newline
Let $p\geq 1$, and take a non-zero element $x$ of the kernel. Let $\leq$ be a lattice order. Now apply the construction of lemma \ref{Pwedge} to obtain $S'$ and $T'$ and a non-zero $y\in\ker\delta'$. Since $\#T'<\#T$, we have $\#T'\leq p-1$. Applying the induction hypothesis we get a contradiction.
\end{proof}
Note that this also follows from Green's Linear syzygy theorem \cite[Theorem 7.1]{eisenbud} applied to the graded module $\bigoplus_{q\geq 0}qS+T$ over the graded ring $\bigoplus_{q\geq 0}qS$. We give this direct proof because we rely on the same technique later.\newline
We now want to construct elements of the kernel of $\delta$ when $p=\#T-1$.
To this end we do the following construction:
let $X=\{P\in\ZZ^n|P+T\subseteq S\}$ and consider the elements of $X$ as variables. To any monomial $A=\prod_{P\in X}P^{a_P}$ of degree $p$ with variables in $X$ we will associate an element $x_A$ of the kernel of $\delta$.

Let $P_1,\ldots,P_p$ be a list of points in $X$ such that each point $P$ occurs $a_P$ times in the list. Let $Q_1,\ldots,Q_{p+1}$ be a list of all points of $T$. Now we define
$$x_A=\prod_P\frac{1}{a_P!}\sum_{\sigma\in S_{p+1}}sgn(\sigma)(P_1+Q_{\sigma(1)})\wedge\ldots\wedge(P_p+Q_{\sigma(p)})\otimes Q_{\sigma(p+1)}.$$
Up to sign, this will be independent of the choice of lists $P_1,\ldots,P_p$ and\newline $Q_1,\ldots,Q_{p+1}$.
\begin{lemma}
The $x_A$ we just constructed has integer coefficients and is in the kernel of $\delta$.
\end{lemma}
\begin{proof}
Consider permutations $\sigma,\sigma'\in S_{p+1}$ as in the previous definition such that $\sigma(\{i|P_i=P\})=\sigma'(\{i|P_i=P\})$ for all $P\in X$. We claim that the terms in $x_A$ corresponding to $\sigma$ and $\sigma'$ will be equal. This is because in the wedge product the only thing that changes is the order, and the change in sign caused by changing the order is compensated by the change in $sgn(\sigma)$.\newline
Now the number of bijections $\sigma'$ with the property that $\sigma(\{i|P_i=P\})=\sigma'(\{i|P_i=P\})$, $\forall P$ is equal to $\prod_Pa_P!$, hence the expression will have integer coefficients.\newline
We now prove that $x_A$ is in the kernel of $\delta$. Obviously the sums $P_i+Q_{\sigma(i)}$ are all in $S$. We claim that when applying $\delta$ everything cancels. Let $C$ be the set of all ordered pairs $(\sigma,i)$ where $\sigma\in S_{p+1}$ and $i\in\{1,\ldots,p\}$. Then
\begin{align}
&\big(\prod_Pa_P!\big)\delta(x_A)\nonumber \\
&=\sum_{(\sigma,i)\in C}(-1)^isgn(\sigma)(P_1+Q_{\sigma(1)})\wedge\ldots\wedge\widehat{(P_i+Q_{\sigma(i)})}\wedge\ldots\wedge(P_p+Q_{\sigma(p)}))\nonumber \\
&\otimes(Q_{\sigma(p+1)}+Q_{\sigma(i)}+P_i).\nonumber
\end{align}
We now partition $C$ into (unordered) pairs: $(\sigma,i)$ and $(\sigma',i')$ belong to the same pair if either they are equal or the following conditions are met:
\begin{itemize}
\item $i=i'$
\item $\sigma(j)=\sigma'(j)$, for all $j\notin\{i,p+1\}$
\item $\sigma(i)=\sigma'(p+1)$ and $\sigma'(i)=\sigma(p+1)$.
\end{itemize}
These conditions imply that $\sigma'^{-1}\circ \sigma$ is a transposition, and hence has sign -1. One now easily sees that pairs in $C$ yield terms that cancel.
\end{proof}
\noindent\textbf{Example.}\newline
Suppose $S=\{(2,0),(1,1),(1,0),(0,2),(0,1),(0,0)\}$ and $T=\{(1,0),(0,1),(0,0)\}$ and $p=2$. Then $X=\{(1,0),(0,1),(0,0)\}$. Take for example the monomial $A=(1,0)(0,1)$. If we take the lists $(1,0),(0,1)$ and $(1,0),(0,1),(0,0)$ we get: 
\begin{align}
x_A=&(1,0)\wedge(1,1)\otimes(0,1)-(1,0)\wedge(0,2)\otimes(1,0)+(2,0)\wedge(0,2)\otimes(0,0)\nonumber \\ -&(2,0)\wedge(0,1)\otimes(0,1)+(1,1)\wedge(0,1)\otimes(1,0)-(1,1)\wedge(1,1)\otimes(0,0).\nonumber
\end{align}
Of course the last term is zero. Note that each term is of the form $P\wedge Q\otimes((2,2)-P-Q)$ where $P$ belongs to the lower right blue triangle and $Q$ to the upper left one.
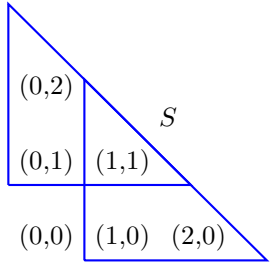
\begin{center}
\begin{tikzpicture}
%\draw [<->] (2.5,0) -- (0,0) -- (0,2.5);
\draw[blue,thick] (.5,-.3) -- (2.9,-.3) -- (.5,2.1) -- (.5,-.3);
\draw[blue,thick] (-.5,.7) -- (-.5,3.1) -- (1.9,.7) -- (-.5,.7);
\node at (1,0) {(1,0)};
\node at (2,0) {(2,0)};
\node at (0,1) {(0,1)};
\node at (0,2) {(0,2)};
\node at (1,1) {(1,1)};
\node at (0,0) {(0,0)};
\node at (1.6,1.6) {$S$};
\end{tikzpicture}
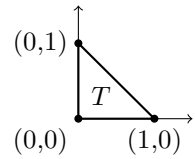
\hspace{100 pt}
\begin{tikzpicture}
\draw [<->] (1.5,0) -- (0,0) -- (0,1.5);
\draw [thick] (0,0) -- (1,0) -- (0,1) -- (0,0);
\fill[black] (0,0) circle (1.5 pt);
\fill[black] (1,0) circle (1.5 pt);
\fill[black] (0,1) circle (1.5 pt);

%\node at (1.8,0) {$x_1$};
%\node at (0,1.7) {$x_2$};

\node at (-.5,-.3) {(0,0)};
\node at (1,-.3) {(1,0)};
\node at (-.5,1) {(0,1)};
\node at (.3,.3) {$T$};
\end{tikzpicture}
\end{center}
\noindent\textbf{Example.}\newline
Suppose $S=\{(0,0),(1,0),(2,0),(3,0),(0,1),(1,1),(2,1),(3,1)\}$ and $T=$\\ $\{(0,0),(1,0),(0,1),(1,1)\}$ so that $X=\{(0,0),(1,0),(2,0)\}$, $p=3$ and consider the monomial $A=(0,0)^2(2,0)$, then we get
\begin{align}
x_A=&-(0,0)\wedge(1,0)\wedge(3,1)\otimes(0,1)+(0,0)\wedge(1,0)\wedge(2,1)\otimes(1,1)\nonumber \\
&-(0,0)\wedge(1,1)\wedge(2,1)\otimes(1,0)+(0,0)\wedge(1,1)\wedge(3,0)\otimes(0,1)\nonumber \\
&-(0,0)\wedge(0,1)\wedge(3,0)\otimes(1,1)+(0,0)\wedge(0,1)\wedge(3,1)\otimes(1,0)\nonumber \\
&-(1,0)\wedge(0,1)\wedge(3,1)\otimes(0,0)+(1,0)\wedge(0,1)\wedge(2,0)\otimes(1,1)\nonumber \\
&-(1,0)\wedge(1,1)\wedge(2,0)\otimes(0,1)+(1,0)\wedge(1,1)\wedge(2,1)\otimes(0,0)\nonumber \\
&-(0,1)\wedge(1,1)\wedge(3,0)\otimes(0,0)+(0,1)\wedge(1,1)\wedge(2,0)\otimes(1,0)\nonumber
\end{align}
In this case the first two wedge factors in each term are from the left blue square and the third wedge factor is from the right blue square. In the definition there are 24 terms but each term occurs twice and we divide by two so only twelve terms are left.
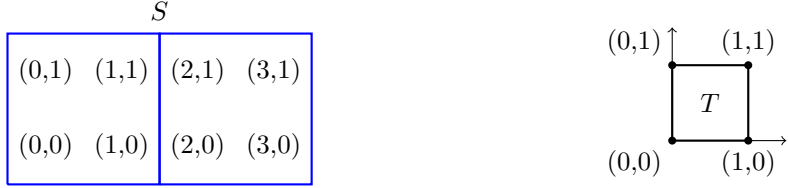
\begin{center}
\begin{tikzpicture}
%\draw [<->] (2.5,0) -- (0,0) -- (0,2.5);
\draw[blue,thick] (-.5,-.5) -- (1.5,-.5) -- (1.5,1.5) -- (-.5,1.5) -- (-.5,-.5);
\draw[blue,thick] (1.5,-.5) -- (3.5,-.5) -- (3.5,1.5) -- (1.5,1.5) -- (1.5,-.5);
\node at (0,0) {(0,0)};
\node at (1,0) {(1,0)};
\node at (2,0) {(2,0)};
\node at (3,0) {(3,0)};
\node at (0,1) {(0,1)};
\node at (1,1) {(1,1)};
\node at (2,1) {(2,1)};
\node at (3,1) {(3,1)};
\node at (1.5,1.8) {$S$};
\end{tikzpicture}
\hspace{100 pt}
\begin{tikzpicture}
\draw [<->] (1.5,0) -- (0,0) -- (0,1.5);
\draw [thick] (0,0) -- (1,0) -- (1,1) -- (0,1) -- (0,0);
\fill[black] (0,0) circle (1.5 pt);
\fill[black] (1,0) circle (1.5 pt);
\fill[black] (0,1) circle (1.5 pt);
\fill[black] (1,1) circle (1.5 pt);

\node at (-.5,-.3) {(0,0)};
\node at (1,-.3) {(1,0)};
\node at (-.5,1.3) {(0,1)};
\node at (1,1.3) {(1,1)};
\node at (.5,.5) {$T$};
\end{tikzpicture}
\end{center}
%If we take the monomial $A=(0,0)^2$ then every term will occur twice but we also divide by two and we get
%$$x_A=(1,0)\wedge(0,1)\otimes(0,0)-(1,0)\wedge(0,0)\otimes(0,1)+(0,1)\wedge(0,0)\otimes(1,0).$$
One way to get rid of the factor $\prod_P\frac{1}{a_P!}$ is to only sum over one element of each equivalence class, where two permutations $\sigma,\sigma'$ are equivalent if $\sigma'(\{i|P_i=P\})=\sigma'(\{i|P_i=P\})$ for all $P\in X$. So the construction works over any field.
\begin{proposition} \label{wedgeindependent}
The $x_A$ for distinct monomials $A$ are linearly independent and the support of any linear combination of them is the convex hull of the union of the $supp(x_A)$ occurring with non-zero coefficient.
\end{proposition}
\begin{proof}
By induction on $p$. In the case $p=0$ there is only one monomial namely the constant monomial 1. The corresponding $x_A$ is $\wedge(\emptyset)\otimes Q$ where $Q$ is the unique point of $T$. So the statement is obvious. So suppose $p\geq 1$. Let $x=\sum_j\lambda_jx_{A_j}$ with the $A_j$ distinct. We have to prove that $$supp(x)=\text{conv}\Big(\bigcup_jsupp(x_{A_j})\Big)\neq\emptyset.$$
To prove this equality, it is enough to prove that every linear map $L:\RR^n\rightarrow\RR$ attains the same maximum on both sides of the equation. It is enough to show this when $L|_{\ZZ^n}$ is injective (as these $L$ are dense). Given such an $L$, let $\leq$ be the order it induces on $\ZZ^n$. Let $Q_M$ be the maximum of $T$ for this order and $P_M\in X$ the greatest point occurring as a variable in some monomial $A_j$. We will prove that $P_M+Q_M$ is the maximum of both sides of the equation, proving that $L$ attains the same maximum on both, and that both sides are non-empty. Obviously nothing greater than $P_M+Q_M$ can possibly occur in any $supp(x_{A_j})$. We have
$$x=\underset{P_M|A_j}{\sum_j}\lambda_jx_{A_j}+\underset{P_M\text{ does not occur in }A_j}{\sum_j}\lambda_jx_{A_j}.$$
For any $A_j$ containing $P_M$ we define $B_j$ as the monomial $A_j/P_M$ of degree $p-1$. Using $T'=T\backslash\{Q_M\}$ we can associate to any $B_j$ an element $x_{B_j}$ so that $x_{A_j}=\pm(P_M+Q_M)\wedge x_{B_j}$ plus terms where everything in the $\wedge$ part is smaller than $P_M+Q_M$. By induction $x_{B_j}\neq 0$ so $P_M+Q_M$ is the maximum of $supp(x_{A_j})$. Finally
$$x=(P_M+Q_M)\wedge\underset{P_M|A_j}{\sum_j}\pm\lambda_jx_{B_j}+\text{ terms without }P_M+Q_M\text{ in the }\wedge\text{ part}.$$
By induction the linear combination of the $x_{B_j}$ is not zero so $P_M+Q_M$ is the maximum of $supp(x)$.
\end{proof}
So far, we have studied the kernel of the map
$$\delta:\bigwedge^pS\otimes T\rightarrow\bigwedge^{p-1}S\otimes(S+T):$$
$$P_1\wedge\ldots\wedge P_p\otimes Q\longmapsto\sum_{i=1}^p(-1)^iP_1\wedge\ldots\wedge\hat{P_i}\wedge\ldots\wedge P_p\otimes (Q+P_i).$$
We now introduce the following maps
$$\delta_i:S^{\otimes p}\otimes T\rightarrow S^{\otimes (p-1)}\otimes(S+T):$$
$$P_1\otimes\ldots\otimes P_p\otimes Q\mapsto P_1\otimes\ldots\otimes\hat{P_i}\otimes\ldots\otimes P_p\otimes(P_i+Q).$$
This time we look at the intersection of the kernels of the $\delta_i$. If $p=0$ there is nothing to intersect so we put $\bigcap_{i=1}^0\ker\delta_i=T$.\newline We introduce this new machinery because it helps us prove our main result.\\

\noindent\textbf{Example.} Let $S=T=\{(1,1),(1,0),(0,1),(0,0)\}$ then
\begin{align}
x=&(1,0)\otimes(0,1)\otimes(0,0)-(1,0)\otimes(0,0)\otimes(0,1)\nonumber \\
-&(0,0)\otimes(0,1)\otimes(1,0)+(0,0)\otimes(0,0)\otimes(1,1)\nonumber
\end{align}
is in $\ker\delta_1\cap\ker\delta_2$.
\begin{proposition}
There is an injective map
$$\iota:\bigwedge^p S\otimes T\rightarrow S^{\otimes p}\otimes T:\hspace{5 pt}
P_1\wedge\ldots\wedge P_p\otimes Q\mapsto\sum_{\sigma\in S_p}sgn(\sigma)P_{\sigma(1)}\otimes\ldots\otimes P_{\sigma(p)}\otimes Q$$
and $\iota(\ker\delta)\subseteq\bigcap_{i=1}^p\ker\delta_i$
\end{proposition}
\begin{proof}
Note that the definition of $\iota$ does not depend on any choices. It is injective as cancellation is impossible. Let us prove the last assertion. We define
$$g:\bigwedge^{p-1}S\otimes(S+T)\rightarrow S^{\otimes(p-1)}\otimes(S+T)$$
analogously to $\iota$. If we can prove that $\delta_i\circ\iota=(-1)^ig\circ\delta$ for all $i$ then it follows that $\iota(\ker\delta)\subseteq\bigcap_i\ker\delta_i$. Let $x=P_1\wedge\ldots\wedge P_p\otimes Q$, we compute
\begin{align}
\delta_i(\iota(x))&=\sum_{\sigma\in S_p}sgn(\sigma)P_{\sigma(1)}\otimes\ldots\otimes\widehat{P_{\sigma(i)}}\otimes\ldots\otimes P_{\sigma(p)}\otimes(P_{\sigma(i)}+Q)\nonumber \\
g(\delta(x))&=\sum_{j=1}^p\sum_{\tau\in S_{p-1}}(-1)^jsgn(\tau)P_{\tau'(1)}\otimes\ldots\otimes P_{\tau'(p-1)}\otimes(P_j+Q),\nonumber
\end{align}
where $\tau'=(j$ $j+1\ldots p)\circ\tau$. Here $(j$ $j+1\ldots p)$ maps every number from $j$ up to $p-1$ to itself plus one and everything smaller than $j$ to itself, and $p$ is mapped to $j$. For every $\tau\in S_{p-1}$ we formally put $\tau(p)=p$ so that $S_{p-1}\subseteq S_p$. There is a bijection from $S_p$ to $\{1,\ldots,p\}\times S_{p-1}$ mapping $\sigma$ to $(j,\tau)$ where $$\sigma\circ(i\hspace{4 pt}i+1\ldots p)=(j\hspace{4 pt}j+1\ldots p)\circ\tau.$$ Note that $sgn(\sigma)(-1)^{p-i}=sgn(\tau)(-1)^{p-j}$. Using this bijection one sees that $\delta_i(\iota(x))=(-1)^ig(\delta(x))$. This proves that $\delta_i\circ\iota=(-1)^ig\circ\delta$ and hence that $\iota(\ker\delta)\subseteq\ker\delta_i$.
\end{proof}
For any sequence $P_1,\ldots,P_p$ of points in $X=\{P|P+T\subseteq S\}$ one defines an element of $\bigcap_{i=1}^p\ker\delta_i$:
$$x_{P_1,\ldots,P_p}:=\sum_{\sigma\in S_{p+1}}sgn(\sigma)(P_1+Q_{\sigma(1)})\otimes\ldots\otimes(P_p+Q_{\sigma(p)})\otimes Q_{\sigma(p+1)}$$
where $Q_1,\ldots,Q_{p+1}$ is a list of all the points of $T$. Whenever we use this notation we assume that $\#T=p+1$. Just as the $x_A$ are intended to be a basis of $\ker\delta$, the $x_{P_1,\ldots,P_p}$ are intended to be a basis of $\bigcap_{i=1}^p\ker\delta_i$.
\begin{lemma}
Consider the right group action of $S_p$ on the set of sequences $P_1,\ldots,P_p\in X$ by permutation. So $P_1,\ldots,P_p\cdot\sigma=P_{\sigma(1)},\ldots,P_{\sigma(p)}$. For a given such sequence let $A$ be the monomial $P_1\ldots P_p$, then
$$\iota(x_A)=\sum_{\bar{\sigma}\in \textup{Stab}(P_1,\ldots,P_p)\backslash S_p}x_{P_{\sigma(1)},\ldots,P_{\sigma(p)}}.$$
(We sum over the right cosets of the stabilizer.)
\end{lemma}
\begin{proof}
It is enough to prove this equality in characteristic zero. We have
\begin{align}
&\iota(x_A)\prod_Pa_P!\nonumber \\
=&\iota\Big(\sum_{\sigma\in S_{p+1}}sgn(\sigma)(P_1+Q_{\sigma(1)})\wedge\ldots\wedge(P_p+Q_{\sigma(p)})\otimes Q_{\sigma(p+1)}\Big) \nonumber \\
=&\sum_{\tau\in S_p}\sum_{\sigma\in S_{p+1}}sgn(\sigma\circ\tau)(P_{\tau(1)}+Q_{\sigma(\tau(1))})\otimes\ldots\otimes(P_{\tau(p)}+Q_{\sigma(\tau(p))})\otimes Q_{\sigma(p+1)}\nonumber \\
=&\sum_{\tau\in S_p}\sum_{\sigma'\in S_{p+1}}sgn(\sigma')(P_{\tau(1)}+Q_{\sigma'(1)})\otimes\ldots\otimes(P_{\tau(p)}+Q_{\sigma'(p)})\otimes Q_{\sigma'(p+1)}\nonumber \\
=&\sum_{\tau\in S_p}x_{P_{\tau(1)},\ldots,P_{\tau(p)}}\nonumber \\
=&\prod_Pa_P!\sum_{\bar{\tau}\in\text{Stab}(P_1,\ldots,P_p)\backslash S_p}x_{P_{\tau(1)},\ldots,P_{\tau(p)}}\nonumber
\end{align}
The last equality follows because the order of the stabilizer is $\prod_Pa_P!$. The result follows by removing the factor $\prod_Pa_P!$.
\end{proof}
\begin{lemma}\label{tensortowedge}
The span of the $x_{P_1,\ldots,P_p}$ intersected with $\iota(\ker\delta)$ is generated by the $\iota(x_A)$. In particular if the $x_{P_1,\ldots,P_p}$ are a basis of $\bigcap_{i=1}^p\ker\delta_i$ then the $x_A$ are a basis of $\ker\delta$.
\end{lemma}
\begin{proof}
We have a right group action of $S_p$ on $S^{\otimes p}\otimes T$, restricting to one on $\bigcap_{i=1}^p\ker\delta_i$: any $\sigma\in S_p$ maps $P_1\otimes\ldots\otimes P_p\otimes Q$ to $sgn(\sigma)P_{\sigma(1)}\otimes\ldots\otimes P_{\sigma(p)}\otimes Q$. Clearly any element of $\iota(\ker\delta)$ is fixed by this action. The action of $S_p$ on the set of sequences $P_1,\ldots,P_p$ in $X$ from the previous lemma is compatible with the action on $S^{\otimes p}\otimes T$ in the sense that $x_{P_1,\ldots,P_p}\cdot\sigma=x_{P_{\sigma(1)},\ldots,P_{\sigma(p)}}$. Choose an element $R_j=P_{1,j},\ldots,P_{p,j}$ out of each orbit of the action on sequences.\newline Consider an element $x$ of $\iota(\ker\delta)$ that is a linear combination of the $x_{P_1,\ldots,P_p}$. We prove that it is generated by the $\iota(x_A)$. Write it as a linear combination of the
$x_{P_1,\ldots,P_p}$:
$$x=\sum_j\sum_{\bar{\sigma}\in\text{Stab}(R_j)\backslash S_p}\lambda_{j,\bar{\sigma}}x_{P_{\sigma(1),j},\ldots,P_{\sigma(p),j}}.$$
Applying the action on $S^{\otimes p}\otimes T$ to this expression permutes the $\lambda_{j,\bar{\sigma}}$, for each $j$. Since $x$ is fixed by the action of $S_p$ and the $x_{P_1,\ldots,P_p}$ are linearly independent (by lemma \ref{independent} below), $\lambda_{j,\bar{\sigma}}$ doesn't depend on $\sigma$. So $x$ is a linear combination of the
$$\sum_{\bar{\sigma}\in\text{Stab}(R_j)\backslash S_p}x_{P_{\sigma(1),j},\ldots,P_{\sigma(p),j}}=\iota(x_A),$$
where $A=P_{1,j}\ldots P_{p,j}$. Note that we used the previous lemma in the last equality. This proves the first assertion. The second assertion follows from the first, the injectivity of $\iota$ and proposition \ref{wedgeindependent}.
\end{proof}
Having established a connection between the $\bigwedge^p S\otimes T$ context and the $S^{\otimes p}\otimes T$ context, we now focus on the latter.
\begin{definition}
For any $x\in\bigcap_{i=1}^p\ker\delta_i$ and $i\in\{1,\ldots,p\}$ we define $supp_i(x)$ to be the convex hull of the set of lattice points occurring in the $i$-th factor of some term of $x$.
\end{definition}
The following lemma is analogous to lemma \ref{Pwedge}.
\begin{lemma}\label{Ptensor}
Let $x\in\bigcap_{i=1}^p\ker\delta_i$ and let $\leq$ be a lattice pre-order on $\ZZ^n$. Fix an $i$ and suppose $P\in supp_i(x)$ is strictly greater than any other point of $supp_i(x)$. Let $T'$ be the set of non-maximal points in $T$. Let $\delta_1',\ldots,\delta_{p-1}':$ $S^{\otimes(p-1)}\otimes T'\rightarrow S^{\otimes(p-2)}\otimes(S+T')$ be defined analogously to $\delta_1,\ldots,\delta_p$. Write
$$x=\sum_j\lambda_j P_{1,j}\otimes\ldots\otimes P(\textup{place }i)\otimes\ldots\otimes P_{p,j}\otimes Q_j+\textup{terms not having }P\textup{ at place }i$$
$$y=\sum_j\lambda_j P_{1,j}\otimes\ldots\otimes P_{p,j}\otimes Q_j\textup{ (with }P\textup{ removed from place }i).$$
Then $y\in\bigcap_{j=1}^{p-1}\ker\delta_j'$.
\end{lemma}
We omit the proof since it is analogous to that of lemma \ref{Pwedge}.\vspace{4 pt}\newline
\textbf{Example.} Let $S=T=\{(1,1),(1,0),(0,1),(0,0)\}$ then
\begin{align}
x=&(1,0)\otimes(0,1)\otimes(0,0)-(1,0)\otimes(0,0)\otimes(0,1)\nonumber \\
-&(0,0)\otimes(0,1)\otimes(1,0)+(0,0)\otimes(0,0)\otimes(1,1)\nonumber
\end{align}
is in $\ker\delta_1\cap\ker\delta_2$. If we take the order coming from the linear map $L(x_1,x_2)=x_2$ then $P_M=(0,1)$ is the unique maximum of $supp_2(x)=\{(0,1),(0,0)\}$. Applying the lemma with $i=2$ we get $y=(1,0)\otimes(0,0)-(0,0)\otimes(1,0)$.
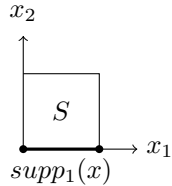
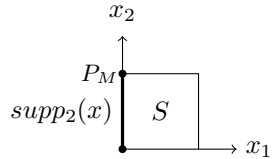
\begin{center}
\begin{tikzpicture}
\draw [<->] (1.5,0) -- (0,0) -- (0,1.5);
\draw (1,0) -- (1,1) -- (0,1);
\draw[very thick] (1,0) -- (0,0);
\fill[black] (1,0) circle (1.5 pt);
\fill[black] (0,0) circle (1.5 pt);
\node at (.5,-.3) {$supp_1(x)$};
\node at (1.8,0) {$x_1$};
\node at (0,1.8) {$x_2$};
\node at (.5,.5) {$S$};
\draw [<->] (8.5,0) -- (7,0) -- (7,1.5);
\draw (8,0) -- (8,1) -- (7,1);
\draw[very thick] (7,1) -- (7,0);
\fill[black] (7,1) circle (1.5 pt);
\fill[black] (7,0) circle (1.5 pt);
\node at (6.2,.5) {$supp_2(x)$};
\node at (8.8,0) {$x_1$};
\node at (7,1.8) {$x_2$};
\node at (7.5,.5) {$S$};
\node at (6.7,1) {\small $P_M$};
\end{tikzpicture}
\end{center}
\begin{lemma}\label{Tleqptensor}
If $p\geq\#T$ then $\bigcap_{i=1}^p\ker\delta_i=0$.
\end{lemma}
Again the proof is analogous to that of theorem \ref{Tleqp}.
\begin{lemma} \label{independent}
Let $\delta_i:S^{\otimes p}\otimes T\rightarrow S^{\otimes(p-1)}\otimes(S+T)$ be the usual maps, $p=\#T-1$. Let $x=\sum_j\lambda_jx_{P_{1,j},\ldots,P_{p,j}}$ be a linear combination with non-zero coefficients then $supp_i(x)$ is the convex hull of $\bigcup_jsupp_i(x_{P_{1,j},\ldots,P_{p,j}})$. In particular the $x_{P_1,\ldots,P_p}$ are linearly independent.
\end{lemma}
One can prove this with the same technique as proposition \ref{wedgeindependent}.
\begin{lemma}\label{suppminsupp}
Let $\Delta$ be a lattice polytope of dimension at least two with $p+1$ lattice points and let $\delta_i:(\ZZ^n)^{\otimes p}\otimes(\Delta\cap\ZZ^n)\rightarrow(\ZZ^n)^{\otimes(p-1)}\otimes\ZZ^n$ be the usual maps. Then for all $x\in\bigcap_{i=1}^p\ker\delta_i\backslash\{0\}$ and $i\in\{1,\ldots,p\}$ we have $$\Delta-\Delta\subseteq supp_i(x)-supp_i(x).$$
\end{lemma}
Note that the lemma can fail if $\Delta$ is one-dimensional. If $n=1$, $\Delta=[0,2]$ and $x=v_0\otimes v_0\otimes v_2-v_0\otimes v_1\otimes v_1-v_1\otimes v_0\otimes v_1+v_1\otimes v_1\otimes v_0$ then $supp_1(x)=[0,1]$ and the conclusion of the lemma fails.
\begin{proof}
We prove this by induction on $p$. We can suppose $i=1$ without loss of generality. Let $x\in\bigcap_{i=1}^p\ker\delta_i\backslash\{0\}$ and take $v\in(\Delta-\Delta)\backslash(supp_1(x)-supp_1(x))$ with integer coordinates. Let $P_1,P_2\in\Delta\cap\ZZ^n$ with $P_2-P_1=v$. Using a unimodular transformation we can suppose $P_1=(0,0,\ldots,0)$ and $P_2=v=(d,0,\ldots,0)$ for some $d>0$.\newline
\underline{Case 1: $\Delta\backslash[P_1,P_2]$ contains more than one lattice point.}\newline Take a linear map $L:\RR^n\rightarrow\RR$ that does not attain a maximum at $P_1$ or $P_2$ on $\Delta$. We take $L$ to attain its maximum on $supp_2(x)$ at only one point. This induces a lattice pre-order $\leq$. We apply lemma \ref{Ptensor} for place 2 to obtain $S'$, $T'$ and a non-zero $y\in\bigcap_{j=1}^{p-1}\ker(\delta'_j)$ with $supp_1(y)\subseteq supp_1(x)$, so $[P_1,P_2]-[P_1,P_2]\nsubseteq supp_1(y)-supp_1(y)$.
If $\#T'=p$ we get a contradiction with the induction hypothesis and if $\#T'<p$ we get a contradiction with lemma \ref{Tleqptensor}. We needed the fact that $\Delta\backslash[P_1,P_2]$ contains more than one lattice point to ensure that $T'$ is of dimension at least two, so we can apply the induction hypothesis.\newline
\underline{Case 2: $\Delta\backslash[P_1,P_2]$ contains only one lattice point.}\newline We can suppose this lattice point is $(0,-1,0,\ldots,0)$. Note that $p=\#T-1=d+1$. Define $L_1,L_2:\ZZ^n\rightarrow\ZZ$ as follows: $L_1((x_1,x_2,\ldots))=-x_1$, $L_2((x_1,x_2,\ldots))=x_1-dx_2$. We claim that $L_1$ (resp.\ $L_2$) attains its maximum on $supp_1(x)$ at more than one lattice point of $supp_1(x)$. Indeed, applying lemma \ref{Ptensor} with $L_1$ (resp.\ $L_2$) and place 1 we get $T'=\{(1,0,0,\ldots,0),\ldots,(d,0,0,\ldots,0)\}$ (resp. $\{(0,0,\ldots,0),\ldots,(d-1,0,0,\ldots,0)\}$). In each case the $y$ we obtain leads to a contradiction with lemma \ref{Tleqptensor}, unless $L_1$ (resp.\ $L_2$) does not attain its maximum on $supp_1(x)$ at a unique point of $supp_1(x)$ (in which case we cannot apply lemma \ref{Ptensor}). Therefore $L_1$ (resp.\ $L_2$) attains its maximum at more than one lattice point of $supp_1(x)$.\newline
\underline{Case 2a: $n=2$.}\newline
Let $(x_1,y_1)$ and $(x_1,y_1+1)$ be points of $supp_1(x)$ on which $L_1$ reaches its maximum and let $(x_2,y_2)$ and $(x_2+d,y_2+1)$ be points of $supp_1(x)$ on which $L_2$ reaches its maximum. We know by maximality that $x_2\geq x_1$ and $L_2((x_1,y_1))\leq L_2((x_2,y_2))$. If $y_2\leq y_1$ then $[(x_2,y_2+1),(x_2+d,y_2+1)]\subseteq\text{supp}_1(x)$ and similarly if $y_2\geq y_1$ then $[(x_1,y_1+1),(x_1+d,y_1+1)]\subseteq\text{supp}_1(x)$. In any case we get a contradiction with the fact that $v=(d,0)\notin supp_1(x)-supp_1(x)$.\\
\begin{center}
\begin{tikzpicture}
\draw[thin] (0,-.75) -- (0,1.5);
\draw[thin] (-.75,-.75) -- (4.5,1);
\draw[thin] (0,.25) -- (-1,.25);
\draw[thin] (-1,.25) -- (-.88,.37);
\draw[thin] (-1,.25) -- (-.88,.13);
\draw[thin] (2.4,.3) -- (2.7,-.6);
\draw[thin] (2.7,-.6) -- (2.78,-.44);
\draw[thin] (2.7,-.6) -- (2.54,-.52);
\fill[black] (0,1) circle (1.5 pt);
\fill[black] (0,0.5) circle (1.5 pt);
\fill[black] (1.5,0) circle (1.5 pt);
\fill[black] (3,.5) circle (1.5 pt);
\node at (-1,.5) {\small $\nabla L_1$};
\node at (3.2,-.6) {\small $\nabla L_2$};
\node at (.6,.5) {\small $(x_1,y_1)$};
\node at (.9,1) {\small $(x_1,y_1+1)$};
\node at (1.5,-.37) {\small $(x_2,y_2)$};
\node at (3.7,.15) {\small $(x_2+d,y_2+1)$};
\end{tikzpicture}
\end{center}
\noindent\underline{Case 2b: $n\geq 3$.}\newline Let $\pi:\ZZ^n\rightarrow\ZZ^{n-2}$ be the projection that deletes the first two coordinates. So $\pi(\Delta)=0$. Now $\delta_1=\sum_{P\in\ZZ^{n-2}}\delta_P$ where $\delta_P$ maps $Q_1\otimes\ldots\otimes Q_p\otimes Q$ to
$$Q_2\otimes\ldots\otimes Q_p\otimes(Q+Q_1)\text{ if }\pi(Q_1)=P\text{ and to zero otherwise}.$$
As $\pi(Q+Q_1)=\pi(Q)+\pi(Q_1)=P$, there can't be any cancellation between $\delta_P(x)$ for different $P$. Therefore $\delta_P(x)=0$ for all $P\in\ZZ^{n-2}$.\newline
For any $P\in\ZZ^{n-2}$ and $Q_0\in\pi^{-1}(0)$ we define a linear automorphism
\begin{align}
\alpha_{P,Q_0}&:(\ZZ^n)^{\otimes p}\otimes T\longrightarrow(\ZZ^n)^{\otimes p}\otimes T\nonumber \\
&:Q_1\otimes\ldots\otimes Q_p\otimes Q\mapsto\begin{cases}
(Q_1+Q_0)\otimes\ldots\otimes Q_p\otimes Q & \text{if }\pi(Q_1)=P \\
Q_1\otimes\ldots\otimes Q_p\otimes Q & \text{else.}
\end{cases}\nonumber
\end{align}
(Recall that by $\ZZ^n$ (resp.\ $T$) we mean the vector space with $\ZZ^n$ (resp.\ $T$) as a basis. So we define the linear map on basis elements and linearly extend them over the base field.) For any $P,P'\in\ZZ^{n-2}$ and $Q_0$ we define
\begin{align}
\alpha'_{P,P',Q_0}&:(\ZZ^n)^{\otimes (p-1)}\otimes T\longrightarrow(\ZZ^n)^{\otimes (p-1)}\otimes T\nonumber \\
&:Q_1\otimes\ldots\otimes Q_{p-1}\otimes Q\mapsto\begin{cases}
Q_1\otimes\ldots\otimes Q_{p-1}\otimes(Q+Q_0) & \text{if }P=P' \\
Q_1\otimes\ldots\otimes Q_{p-1}\otimes Q & \text{else.}
\end{cases}\nonumber
\end{align}
Then $\alpha_{P,P',Q_0}'\circ\delta_{P'}=\delta_{P'}\circ\alpha_{P,Q_0}$, from which it follows that $\alpha_{P,Q_0}(x)$ is in $\ker(\delta_{P'})$ for all $P'\in\ZZ^{n-2}$ and $Q_0\in\pi^{-1}(0)$. So $\alpha_{P,Q_0}(x)\in\ker\delta_1$ and if we define
\begin{align}
\alpha''_{P,Q_0}&:(\ZZ^n)^{\otimes{p-1}}\otimes T\longrightarrow(\ZZ^n)^{\otimes{p-1}}\otimes T\nonumber \\
&:Q_1\otimes\ldots\otimes Q_{p-1}\otimes Q\mapsto\begin{cases}
(Q_1+Q_0)\otimes\ldots\otimes Q_{p-1}\otimes Q & \text{if }\pi(Q_1)=P \\
Q_1\otimes\ldots\otimes Q_{p-1}\otimes Q & \text{else,}
\end{cases}\nonumber
\end{align}
then $\alpha''_{P,Q_0}\circ\delta_i=\delta_i\circ\alpha_{P,Q_0}$ for all $i=2,\ldots, p$. Therefore $\alpha_{P,Q_0}(x)\in\bigcap_{i=1}^p\ker\delta_i$.\\

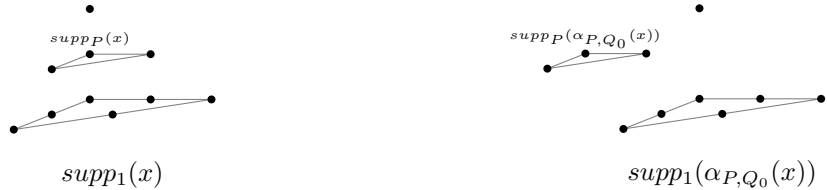
\begin{tikzpicture}
\draw[gray] (0,0) -- (.8,0) -- (1.6,0) -- (.3,-.2) -- (-1,-.4) -- (-.5,-.2) -- (0,0);
\draw[gray] (0,.6) -- (.8,.6) -- (-.5,.4) -- (0,.6);
\fill[black] (0,0) circle (1.5 pt);
\fill[black] (.8,0) circle (1.5 pt);
\fill[black] (1.6,0) circle (1.5 pt);
\fill[black] (-.5,-.2) circle (1.5 pt);
\fill[black] (.3,-.2) circle (1.5 pt);
\fill[black] (-1,-.4) circle (1.5 pt);
\fill[black] (0,.6) circle (1.5 pt);
\fill[black] (.8,.6) circle (1.5 pt);
\fill[black] (-.5,.4) circle (1.5 pt);
\fill[black] (0,1.2) circle (1.5 pt);
\node at (.3,-1) {\normalsize $supp_1(x)$};
\node at (0,.8) {\tiny $supp_P(x)$};
\end{tikzpicture}
\hspace{100 pt}
\begin{tikzpicture}
\draw[gray] (0,0) -- (.8,0) -- (1.6,0) -- (.3,-.2) -- (-1,-.4) -- (-.5,-.2) -- (0,0);
\draw[gray] (-1.5,.6) -- (-.7,.6) -- (-2,.4) -- (-1.5,.6);
\fill[black] (0,0) circle (1.5 pt);
\fill[black] (.8,0) circle (1.5 pt);
\fill[black] (1.6,0) circle (1.5 pt);
\fill[black] (-.5,-.2) circle (1.5 pt);
\fill[black] (.3,-.2) circle (1.5 pt);
\fill[black] (-1,-.4) circle (1.5 pt);
\fill[black] (-1.5,.6) circle (1.5 pt);
\fill[black] (-.7,.6) circle (1.5 pt);
\fill[black] (-2,.4) circle (1.5 pt);
\fill[black] (0,1.2) circle (1.5 pt);
\node at (.3,-1) {\normalsize $supp_1(\alpha_{P,Q_0}(x))$};
\node at (-1.5,.8) {\tiny $supp_P(\alpha_{P,Q_0}(x))$};
\end{tikzpicture}\\

For any $P\in\ZZ^{n-2}$ we define $supp_P(x)$ as the convex hull of all points in $\pi^{-1}(P)$ that occur in the first factor of some term of $x$. Of course this is at most two-dimensional. If we can prove that whenever this support is non-empty it contains more than 1 point where $L_1$ (resp.\ $L_2$) attains its maximum, then we can perform the same reasoning as in case 2a on $supp_P(x)$ to obtain a contradiction. If we choose a $Q_0$ with $L_1(Q_0)$ (resp.\ $L_2(Q_0)$) high enough, then $L_1$ (resp.\ $L_2$) will attain its maximum on $supp_1(\alpha_{P,Q_0}(x))$ only at points of $supp_P(\alpha_{P,Q_0}(x))$. This means that there are at least two points of $supp_P(\alpha_{P,Q_0}(x))$ where $L_1$ (resp.\ $L_2$) attains its maximum. Since $supp_P(\alpha_{P,Q_0}(x))=Q_0+supp_P(x)$ the same is true for $supp_P(x)$, so we are done.
\end{proof}
\begin{theorem}
If $S=\Delta'\cap\ZZ^n$, $T=\Delta\cap\ZZ^n$ and $p=\#T-1$ with $\Delta,\Delta'$ convex and $\Delta$ a bounded lattice polytope of dimension greater than one, then the expressions $x_{P_1,\ldots,P_p}$ with $P_i\in X:=\{P|P+T\subseteq S\}$ are a basis of $\bigcap_{i=1}^p\ker\delta_i$ and hence the $x_A$ for monomials $A$ of degree $p$ with variables in $X$ are a basis of $\ker\delta$.
\end{theorem}
\begin{proof}
Let $H$ be the set of all bounded convex lattice polytopes in $\ZZ^n$ that are either of dimension greater than one or have just two lattice points. By lemma \ref{tensortowedge} we only have to prove the first statement. In fact we only have to prove that the $x_{P_1,\ldots,P_p}$ generate $\bigcap_{i=1}^p\ker\delta_i$ by lemma \ref{independent}. We prove it for all $\Delta\in H$ by induction on $p=\#\Delta\cap\ZZ^n-1$.\newline
Suppose first $p=1$, then $T$ has just two points and we have to show that the kernel of $\delta_1:S\otimes T\rightarrow S+T$ is generated by expressions of the form $(P+Q_1)\otimes Q_2-(P+Q_2)\otimes Q_1$ where $T=\{Q_1,Q_2\}$ and $P\in X$. Consider the map $f:S\times T\rightarrow S+T$ of sets given by addition of lattice points, then every point $P'$ of $S+T$ is reached by at most two elements of $S\times T$ namely $(P+Q_1,Q_2)$ and $(P+Q_2,Q_1)$ with $P=P'-Q_1-Q_2$. We can write $S\otimes T$ as the direct sum of the linear span of each $f^{-1}(P)$ with $P\in S+T$. The kernel of $\delta_1$ is the direct sum of the kernels of $\delta_1$ restricted to each span of $f^{-1}(P)$. The result easily follows.
\newline Now for the induction step suppose $p\geq 2$. Let $Q_M$ be any extreme point of $\Delta$ such that $\text{conv}(\Delta\cap\ZZ^n\backslash\{Q_M\})\in H$.
Using some unimodular transformations one can squeeze $\Delta$ into $(\RR_{\geq 0})^n$ in such a way that $Q_M=(x_M,0,\ldots,0)$ and all other points of $\Delta$ have first coordinate smaller than $x_M$. We can also make sure that the smallest first coordinate in $\Delta$ is zero.\newline
(One can do all this as follows: first one chooses a linear form $L:\ZZ^n\rightarrow \ZZ$ that attains its maximum on $\Delta$ only at $Q_M$. One can choose $L$ with integer coefficients with no prime factors that they all share. One then chooses a unimodular transformation $U_1:\ZZ^n\overset{\cong}{\rightarrow}\ZZ^n$ whose first component is $L$. Then $U_1(Q_M)$ has its first coordinate greater than that of any other point of $U_1(\Delta)$. Next one chooses a unimodular transformation $U_2$ of the form $(x_1,\ldots,x_n)\mapsto(x_1,x_2-a_2x_1,\ldots,x_n-a_nx_1)$ with $a_2,\ldots,a_n$ large enough so that all the other coordinates of $U_2(U_1(Q_M))$ are smaller than those of the other points of $U_2(U_1(\Delta))$. Finally one uses a translation to map $U_2(U_1(Q_M))$ to $(x_M,0,\ldots,0)$ where $x_M$ is the greatest first coordinate on $U_2(U_1(\Delta))$ minus the smallest.)\newline
\underline{Claim:} It is enough to prove the statement in the case where $\Delta'=\RR^n$.\newline
\underline{Proof.} Suppose it is true for $\Delta'=\RR^n$, we prove it for arbitrary $\Delta'$. If $x\in\bigcap_{i=1}^p\ker\delta_i$ then it is a linear combination of some $x_{P_1,\ldots,P_p}$. By lemma \ref{independent} their supports are contained in the supports of $x$, hence in $\Delta'$.\newline
Henceforth we assume $\Delta'=\RR^n$. We put the lexicographical ordering on $(\ZZ_{\geq 0})^n$, meaning $(x_1,\ldots,x_n)<(x_1',\ldots,x_n')$ if for the smallest $i$ with $x_i\neq x_i'$ we have $x_i<x_i'$.\newline So suppose there exists an $x\in\bigcap_{i=1}^p\ker\delta_i$ that is not a linear combination of the $x_{P_1,\ldots,P_p}$. We can translate the first factor so that $supp_1(x)\subseteq(\RR_{\geq 0})^n$. We take $x$ so that the lexicographic maximum of $supp_1(x)$ is minimal. (We can do this because there are no lexicographic infinite descents in $(\ZZ_{\geq 0})^n$.) We will find a contradiction. Let $P'_M$ be the maximum of $supp_1(x)$ and $e$ its first coordinate. Let $Q_m\in\Delta\cap\ZZ^n$ be some point with first coordinate zero. By lemma \ref{suppminsupp} $Q_M-Q_m\in\Delta-\Delta\subseteq supp_1(x)-supp_1(x)$. It follows that $e\geq x_M$. So $P_M:=P'_M-Q_M\in(\RR_{\geq 0})^n$ because its first coordinate $e-x_M$ is $\geq 0$ and all the other coordinates are equal to those of $P'_M$.
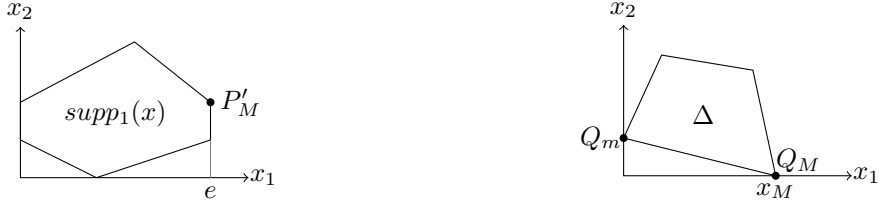
\begin{center}
\begin{tikzpicture}
\draw [<->] (3,0) -- (0,0) -- (0,2);
\draw [thin] (2.5,1) -- (2.5,.5) -- (1,0) -- (0,.5) -- (0,1) -- (1.5,1.8) -- (2.5,1);
\draw[gray, thin] (2.5,.5) -- (2.5,0);
\fill[black] (2.5,1) circle (1.5 pt);
\node at (3.2,0) {$x_1$};
\node at (0,2.2) {$x_2$};
\node at (1.25,.85) {$supp_1(x)$};
\node at (2.9,1) {$P_M'$};
\node at (2.5,-.2) {$e$};
\end{tikzpicture}
\hspace{100 pt}
\begin{tikzpicture}
\draw [<->] (3,0) -- (0,0) -- (0,2);
\draw [thin] (2,0) -- (0,.5) -- (.5,1.6) -- (1.7,1.4) -- (2,0);
\fill[black] (2,0) circle (1.5 pt);
\fill[black] (0,.5) circle (1.5 pt);
\node at (3.2,0) {$x_1$};
\node at (0,2.2) {$x_2$};
\node at (1.05,.8) {$\Delta$};
\node at (2.3,.2) {$Q_M$};
\node at (-.3,.5) {$Q_m$};
\node at (2,-.2) {$x_M$};
\end{tikzpicture}
\end{center}
We now apply lemma \ref{Ptensor} to $x$ to obtain $y\in\bigcap_{i=1}^{p-1}\ker\delta_i'$ where $\delta_i':(\ZZ^n)^{\otimes(p-1)}\otimes T'\rightarrow(\ZZ^n)^{\otimes(p-2)}\otimes\ZZ^n$ are the usual maps and where $T'=T\backslash\{Q_M\}$. This $y$ satisfies
$$x=P'_M\otimes y\text{ plus terms whose first factor is }<P'_M.$$
By induction
$$y=\sum_j\lambda_jy_{P_{1,j},\ldots,P_{p-1,j}},\text{ for some }P_{i,j}\in\ZZ^n.$$
Using the fact that $x_{P_M,P_{1,j},\ldots,P_{p-1,j}}=(P_M+Q_M)\otimes y_{P_{1,j},\ldots,P_{p-1,j}}$ plus terms whose first factor is smaller than $P_M+Q_M=P'_M$ we see that $P'_M$ is the maximum of $supp_1(x')$ where $$x'=\sum_j\lambda_jx_{P_M,P_{1,j},\ldots,P_{p-1,j}}.$$ In $x-x'$ the terms with $P'_M$ cancel so the maximum of $supp_1(x-x')$ is smaller than $P'_M$, contradicting the minimal choice of $x$. (The fact that $P_M\in(\RR_{\geq 0})^n$ is important because it ensures that $supp_1(x-x')\subseteq(\RR_{\geq 0})^n$.)\newline The last assertion of the theorem follows from lemma \ref{tensortowedge}.
\end{proof}
\begin{corollary} \label{corollary}
Let $\Delta$ and $\Delta'$ be convex lattice polytopes with $\Delta$ of dimension greater than 1 and $T=\Delta\cap\ZZ^n$, $S=\Delta'\cap\ZZ^n$ and $p=\#T-1$. If $$\delta:\bigwedge^pS\otimes T\rightarrow\bigwedge^{p-1}S\otimes(S+T)$$ is the usual map then the dimension of $\ker\delta$ is $\binom{p+\#X-1}{p}$ where $X=\{P|P+T\subseteq S\}$.
\end{corollary}
This follows because the number of degree $p$ monomials with variables in $X$ is $\binom{p+\#X-1}{p}$.
We end with the case when $\Delta$ is 1-dimensional. This time the formula works for all $p$.
\begin{theorem} \label{1dim}
Let $\Delta=\textup{conv}((0,0,\ldots,0),(d,0,\ldots,0))$ with $d\geq 0$ and let $\Delta'\subseteq\ZZ^n$ a bounded convex lattice polygon, then for all $p\leq d+1$ the dimension of the kernel of the usual map $\delta$ is $(d-p+1)\binom{\#X}{p}$ where $X=\{P|P+\{(0,\ldots,0),(1,0,\ldots,0)\}\subseteq S\},$ $S=\Delta'\cap\ZZ^n$.
\end{theorem}
\begin{proof} Let $T=\Delta\cap\ZZ^n$. Put a lattice order $\leq$ on $\ZZ^n$ such that $(1,0,\ldots,0)>(0,\ldots,0)$.
Let $I=\{(0,0,\ldots,0),(1,0,\ldots,0)\}$. For any $P_1<\ldots<P_p$ in $X$ and $Q\in T$ with first coordinate in $\{p,\ldots,d\}$ we define
\begin{equation}\label{hypercube}
\sum_{i_1,\ldots,i_p\in I}(-1)^{i_1+\ldots+i_p}(P_1+i_1)\wedge\ldots\wedge(P_p+i_p)\otimes(Q-i_1-\ldots-i_p),
\end{equation}
where we abusively write $(-1)^{i_1+\ldots+i_p}$ for the power of $-1$ whose exponent is the first coordinate of $i_1+\ldots+i_p$. These expressions are in the kernel of $\delta$. We will prove that they are a basis of the kernel which proves the theorem because there are exactly $(d-p+1)\binom{\#X}{p}$ of these. We will do so by induction on $\#S$. The case where $p=0$ is easy as the domain and kernel of $\delta$ are both just $T$ and have the points in $T$ as a basis. So suppose $p\geq 1$ and let $x\in\ker\delta$, we will show that it is a linear combination of expressions like (\ref{hypercube}). Let $P_M$ be the maximum of $S$. If $P_M\notin supp(x)$ we apply the induction hypothesis to $S'=S\backslash\{P_M\}$ and we are done. So assume $P_M\in supp(x)$, then by lemma \ref{Pwedge} we can write $x=P_M\wedge y$ plus terms not containing $P_M$ in the $\wedge$ part. Here $y\in\bigwedge^pS'\otimes T'$ where $T'=\{(0,0,\ldots,0),\ldots,(d-1,0,\ldots,0)\}$. Note also that $P_M-(1,0,\ldots,0)\in S$ as otherwise the terms in $\delta(x)$ where $P_M$ is removed from the $\wedge$ would have nothing to cancel against. This is because these would be the only terms of $\delta(x)$ where the point after the $\otimes$ agrees with $P_M$ in all but the first coordinate. Applying the induction hypothesis to $y$ we get
\begin{align}
y=&\sum_j\lambda_j\sum_{i_1,\ldots,i_{p-1}\in I}(-1)^{{i_1}+\ldots+i_{p-1}}\nonumber \\ &(P_{1,j}+i_1)\wedge\ldots\wedge(P_{p-1,j}+i_{p-1})\otimes(Q_j-i_1-\ldots-i_{p-1}).\nonumber
\end{align}
Therefore $x$ can be written as $(-1)^px'$ plus terms not containing $P_M$ in the $\wedge$ part where
\begin{align}
x'=&\sum_j\lambda_j\sum_{i_1,\ldots,i_p\in I}(-1)^{i_1+\ldots+i_p}\nonumber \\ &(P_{1,j}+i_1)\wedge\ldots\wedge(P_{p-1,j}+i_{p-1})\wedge(P_M-(1,0,\ldots,0)+i_p)\nonumber \\
&\otimes(Q_j+(1,0,\ldots,0)-i_1-\ldots-i_p).\nonumber
\end{align}
So we can apply the induction hypothesis to $x-(-1)^px'$ to conclude that $x$ is a linear combination of expressions like (\ref{hypercube}).\newline Finally, we show linear independence of the expressions, again by induction on $\#S$. the case $p=0$ is again trivial, let  $p\geq 1$. Let $\sum_i\lambda_ix_i$ be a linear combination of our expressions that yields zero. Each $x_i$ containing $P_M$ in its support can be written as $P_M\wedge y_i$ plus terms not containing $P_M$.
Then up to sign $y_i$ is an expression like (\ref{hypercube}) but with $p-1$ instead of $p$ and with the set $S\backslash\{P_M\}$ in stead of $S$. By the induction hypothesis the $y_i$ are linearly independent. But then it follows that $P_M$ cannot occur at all in the wedge part of any $x_i$, otherwise the linear combination could not yield zero. And then one again applies the induction hypothesis with $S\backslash\{P_M\}$ to obtain a contradiction.
\end{proof}


\begin{thebibliography}{99}
\bibitem{nagelaprodu} Marian Aprodu and Jan Nagel, \emph{Koszul cohomology and algebraic geometry}, University Lecture Series \textbf{52}, American Mathematical Society, 125 pp.\ (2010)

\bibitem{brunsconca} Winfried Bruns, Aldo Conca, Tim R\"omer, \emph{Koszul homology and syzygies of Veronese subalgebras}, Mathematische Annalen \textbf{351}(4), pp.\ 761-779 (2011)

\bibitem{bettitoricsurfaces} Wouter Castryck, Filip Cools, Jeroen Demeyer, Alexander Lemmens, \emph{Computing graded Betti tables of toric surfaces}, preprint

\bibitem{betticurves} Wouter Castryck, Filip Cools, Alexander Lemmens, \emph{Canonical 
 syzygies of smooth curves in toric surfaces}, preprint

\bibitem{coxlittleschenck} David Cox, John Little and Hal Schenck, \emph{Toric varieties}, Graduate Studies in Mathematics \textbf{124}, American Mathematical Society (2011)

\bibitem{EinErmanLazarsfeld} Lawrence Ein, Daniel Erman, Robert Lazarsfeld, \emph{A quick proof of nonvanishing for asymptotic syzygies}, Algebraic Geometry \textbf{3}(2), pp.\ 211-222 (2016)

\bibitem{einlazarsfeld} Lawrence Ein, Robert Lazarsfeld, \emph{Asymptotic Syzygies of Algebraic Varieties}, Inventiones Mathematicae \textbf{190}, pp.\ 603-646 (2012)

\bibitem{eisenbud} David Eisenbud, \emph{The geometry of syzygies. A second course in commutative algebra
and algebraic geometry}, Graduate Texts in Mathematics \textbf{229}, Springer-Verlag, New York (2005)

\bibitem{gallego} Francisco J.\ Gallego and Bangere P.\ Purnaprajna, \emph{Some results on rational surfaces and Fano varieties}, Journal f\"ur die Reine und Angewandte Mathematik \textbf{538}, pp.\ 25-55 (2001) 

\bibitem{greco} Ornella Greco and Ivan Martino, \emph{Syzygies of the Veronese modules}, Communications in Algebra \textbf{44}(9), pp. 3890-3906

\bibitem{heringphd} Milena Hering, \emph{Syzygies of toric varieties}, Ph.D.\ thesis, University of Michigan (2006)

\bibitem{loose} Frank Loose, \emph{On the graded Betti numbers of plane algebraic curves}, Manuscripta Mathematica \textbf{64}(4), pp.\ 503-514 (1989)

\bibitem{rubei} Elena Rubei, \emph{A result on resolutions of Veronese embeddings}, Annali dell'Universit\`a di Ferrara Sezione VII. Scienze Matematiche \textbf{50}, pp.\ 151-165 (2004)

\bibitem{ottaviani} Giorgio Ottaviani, Raffaela Paoletti, \emph{Syzygies of Veronese embeddings}, Compositio Mathematica \textbf{125}(1), pp.\ 31-37 (2001)

\bibitem{park} Euisung Park, \emph{On syzygies of Veronese embedding of arbitrary projective varieties}, Journal of Algebra \textbf{322}(1), pp.\ 108-121 (2009)

\bibitem{schenck} Hal Schenck, \emph{Lattice polygons and Green's theorem}, Proceedings of the American Mathematical Society \textbf{132}(12), pp.\ 3509-3512 (2004)

\bibitem{schreyer}  Frank-Olaf Schreyer, \emph{Syzygies of canonical curves and special linear series}, Mathematische Annalen \textbf{275}, pp.\ 105-138 (1986)
\end{thebibliography}
\end{document}